\documentclass[reqno,12pt,hidelinks]{amsart}
\usepackage{amssymb,amsmath,amsthm,amsxtra,calc,bm, color, verbatim,algorithm2e,todonotes, xurl}
\usepackage{enumerate}
\usepackage[margin=.95in]{geometry}

\usepackage[scr=boondoxo]{mathalfa}
\usepackage{mathtools}

\usepackage{mleftright}
\mleftright

\usepackage{nccmath}
\usepackage{cases}

\usepackage{calc}
\usepackage{graphicx}

\usepackage{hyperref}

\usepackage[draft]{microtype}

\makeatletter
\DeclareRobustCommand{\pdot}{\mathbin{\mathpalette\pdot@\relax}}
\newcommand{\pdot@}[2]{%
  \ooalign{%
    $\m@th#1\circ$\cr
    \hidewidth$\m@th#1\cdot$\hidewidth\cr
  }%
}

\newcommand{\ord}{\operatorname{ord}}
\renewcommand{\|}{\big |}

\newcommand{\texpdf}[2]{\texorpdfstring{#1}{#2}}
\def\Z{\mathbb{Z}}
\def\Q{\mathbb{Q}}

\def\H{\mathbb{H}}

\def\C{\mathbb{C}}

\def\SL{{\rm SL}}

\def\GL{{\rm GL}}

\newcommand{\pfrac}[2]{\left(\frac{#1}{#2}\right)}
\newcommand{\ptfrac}[2]{\left(\tfrac{#1}{#2}\right)}
\newcommand{\pMatrix}[4]{\left(\begin{matrix}#1 & #2 \\ #3 & #4\end{matrix}\right)}

 \newcommand{\spmatrix}[4]{\left(\begin{smallmatrix}#1 & #2 \\ #3 & #4\end{smallmatrix}\right)}
\renewcommand{\bar}[1]{\overline{#1}}

\renewcommand{\tilde}{\widetilde}

\renewcommand{\sl}{\big| }

\def\ep{\varepsilon}

\newtheorem{theorem}{Theorem}[section]
\newtheorem{lemma}[theorem]{Lemma}

\newtheorem{proposition}[theorem]{Proposition}

\theoremstyle{remark}
\newtheorem*{remark}{Remark}

\numberwithin{equation}{section}

\mathtoolsset{showonlyrefs}

\def\cpm{c\phi_m}
\def\cp13{c\phi_{13}}

\title{Arithmetic properties of generalized Frobenius partitions}

\date{\today}
\author{Scott Ahlgren}
\address{Department of Mathematics\\
University of Illinois\\
Urbana, IL 61801} 
\email{sahlgren@illinois.edu} 

\author{Cruz Castillo}
\address{Department of Mathematics\\
University of Illinois\\
Urbana, IL 61801} 
\email{ccasti30@illinois.edu}

\thanks{Scott Ahlgren was  partially supported by  grant \#963004 from the Simons Foundation.
Cruz Castillo was partially supported by the Alfred P. Sloan Foundation’s MPHD Program and by the National Science Foundation Graduate Research Fellowship grant DGE 21-46756.
}
\begin{document}
\begin{abstract}
Ramanujan proved three famous congruences for the partition function modulo 5, 7, and 11. The first author and Boylan proved that these congruences are the only ones of this type.   In 1984 Andrews introduced the $m$-colored Frobenius partition functions 
$\cpm$; these are natural higher-level analogues of the partition function which have attracted a great deal of attention in the ensuing decades.
For each $m\in \{5, 7, 11\}$ there are two analogues of Ramanujan's  congruences for $\cpm$, and 
for these   $m$ we prove there are no congruences like Ramanujan's other than these six.
Our methods involve a blend of theory  and computation with modular forms.
\end{abstract}

\maketitle

\setcounter{tocdepth}{2}
\section{Introduction}

Andrews \cite{Andrews} introduced the $m$-colored generalized Frobenius partition functions $\cpm(n)$ 
(whose complete definition is given in Section~\ref{sec:background and preliminary results}) and developed many of their properties. 
The functions $\cpm(n)$ are natural higher-level analogues of the ordinary partition function $p(n)$.
For example we have 
 $c\phi_1(n)=p(n)$ \cite[$(5.15)$]{Andrews}, and for odd $m$ the generating function for $\cpm$ is a modular form of level $m$ and weight $-1/2$ (see \cite{Chan-Wang-Yang} or \eqref{eq:Amdef} below); this generalizes the relationship between the generating function for $p(n)$ and the Dedekind eta function.

Chan, Wang and Yang  \cite{Chan-Wang-Yang} proved a number of modular properties of the generating functions $\sum \cpm(n)q^n$ and studied these functions in detail for $m\leq 17$; we refer the reader to this paper for a more complete description of previous research on this topic.

The functions $\cpm$  enjoy many of the famous congruence properties satisfied by the partition function.   In a recent paper \cite{AAD:frob} the first author, Andersen and Dicks showed\footnote{The hypotheses can be relaxed; see \cite{AAD:frob}.} that if $m\geq 5$
is prime and $\ell$ is a prime with $\ell>m^2$
 then there are infinitely many congruences
 of the form 
 \begin{equation}
 c\phi_m\left(\frac{\ell p^2n+m}{24}\right) \equiv 0 \pmod\ell  \ \ \ \text{ if }
\ \ \ \pfrac{n}{p}=
\pfrac{-1}{p}^{\frac{m\ell+1}{2}}\pfrac{p}{m},
\end{equation}
 where $p\equiv 1\pmod\ell$ is prime.
Examples of such congruences with $\ell\leq 31$ for $p(n)=c\phi_1(n)$ were
discovered in the 1960s by Atkin, and the existence of infinitely many such congruences in this case was proved by the first author, Allen and Tang in 
 \cite{Ahlgren-Allen-Tang}.

For $c\phi_1(n)=p(n)$ we also have the three famous  congruences of Ramanujan:
\begin{gather} \label{eq:ram-cong}
    p\pfrac{\ell n+1}{24} \equiv 0 \pmod{\ell}, \qquad \ell=5,7,11.
\end{gather}
A beautifully symmetric family of 
analogous congruences can be proved for $m\in \{5, 7, 11\}$ using results from \cite{Chan-Wang-Yang} (see \cite[\S 1]{AAD:frob} for details). In particular
we have
\mathtoolsset{showonlyrefs=false}
\begin{align} \label{eq:ram-cong5}
    c\phi_5\pfrac{\ell n+5}{24} &\equiv 0 \pmod{\ell}, \qquad \ell=7,11;\\
 \label{eq:ram-cong7}
    c\phi_7\pfrac{\ell n+7}{24} &\equiv 0 \pmod{\ell}, \qquad \ell=5, 11;\\
\label{eq:ram-cong11}
    c\phi_{11}\pfrac{\ell n+11}{24} &\equiv 0 \pmod{\ell}, \qquad \ell=5, 7.
\end{align}
\mathtoolsset{showonlyrefs=true}
Work of the  first author and Boylan \cite{Ahlgren-Boylan} together with a result of Kiming and Olsson \cite{KimingOlsson} shows that there are no more congruences like Ramanujan's for $p(n)$.  In other words there are no congruences $p(\ell n+\beta)\equiv 0\pmod\ell$ for $\ell\geq 13$.
Our goal here is to establish  the analogous result  for  $c\phi_5, c\phi_7$, and $c\phi_{11}$. In particular we have the following:

\begin{theorem}\label{thm:main}  For $m\in\{5, 7, 11\}$ and  $\ell$ prime there are no congruences other than \eqref{eq:ram-cong5}--\eqref{eq:ram-cong11} of the form
\begin{gather}\label{eq:frob_congruences}
 \cpm(\ell n+\beta)\equiv 0\pmod\ell.
\end{gather}
\end{theorem}
\begin{remark}
   It should be  possible with much more effort to use the techniques of this paper to  prove the nonexistence of congruences \eqref{eq:frob_congruences} for larger values of $m$.  For example, we show in Section~\ref{sec:m=13} that there are no such congruences when $m=13$.
\end{remark}

In the next section we briefly discuss some background material and preliminary lemmas.    In Section~\ref{sec:classify} we adapt the methods of \cite{KimingOlsson} and \cite{Ahlgren-Boylan}
to prove that the existence of a congruence imposes strong restrictions on the theta-cycle of a particular modular form associated to $\cpm$.
In the following sections we describe the substantial computations which are necessary for $m=5, 7, 11$ to prove Theorem~\ref{thm:main} (by contrast the analogous computation in \cite{Ahlgren-Boylan} for $m=1$ is quite simple).  For example, the case $m=11$  (which is the most complicated of the cases we consider) requires the construction of  explicit bases for spaces of modular forms on $\Gamma_0(11)$ which are tailored to the problem, as well as the consideration of more than 13,000 linear systems of dimension $9\times11$. 
We give complete treatments of the cases $m=5, 11$ and a sketch for $m=7,13$.

\section{Background}\label{sec:background and preliminary results}
Here we discuss some background and notation and prove some preliminary lemmas. 
We begin by defining the functions $\cpm$.
The Frobenius symbol associated to a partition of $n$
is an array of non-negative integers
\begin{equation} \label{eq:frob-symb}
    \left( \begin{array}{cccc}
    a_1 & a_2 & \cdots & a_r \\
    b_1 & b_2 & \cdots&b_r\end{array} 
    \right)
\end{equation}
 with strictly decreasing rows and with
\begin{equation} \label{eq:n-sum-frob-symb}
    n=r+\sum^{r}_{i=1}a_i+\sum^{r}_{i=1} b_i.
\end{equation}
An $m$-colored generalized Frobenius symbol
is a  symbol \eqref{eq:frob-symb} in which entries are allowed to repeat up to $m$ times.
Each entry is colored with one of the
 colors $1, 2, \dots, m$
and the rows are strictly decreasing 
with respect to the ordering
\begin{gather*}
    0_1<0_2<\cdots<0_m<1_1<1_2<\cdots<1_m<2_1<2_2<\cdots.
\end{gather*}
Finally, $\cpm(n)$ is the  number of $m$-colored generalized Frobenius symbols which sum to $n$ 
with the formula \eqref{eq:n-sum-frob-symb}.
Andrews \cite[Theorem 5.2]{Andrews} proved that 
\begin{gather}\label{eq:cphidef}
\sum c\phi_m\pfrac{n+m}{24}q^\frac n{24}=\eta^{-m}(z)\sum_{n=0}^\infty r_m(n)q^n,
\end{gather}
where $r_m(n)$ is the number of representations of $n$ by the quadratic form
\begin{equation}\label{eq:r_Qdef}
    \sum_{i=1}^{m-1}x_i^2+\sum_{1\leq i<j\leq m-1}x_ix_j
\end{equation}
and $\eta(z)$ is the Dedekind eta function.

We will also require some facts from the theory of modular forms.
Let $f$ be a function on the upper half-plane $\H$, $k$ be a non-negative even integer,  and $\gamma=\spmatrix abcd\in \GL_2^+(\Q)$. We denote the weight $k$ slash operator by 
\begin{gather}
\left(f\|_k\gamma\right)(z):=(\det\gamma)^\frac k2(cz+d)^{-k}f(\gamma z).
\end{gather}
 Let $A$ be a subring of $\C$, $m$ be a positive integer, and $\chi$ be a Dirichlet character modulo $m$. Let $M_k(m, \chi, A)$  be the $A$-module of functions $f$ on $\H$ which  satisfy the transformation law
\begin{gather}\label{eq:transform}
f\|_k\gamma=\chi(\gamma)f \quad \text{ for all } \quad \gamma \in \Gamma_0(m),
\end{gather}
which are holomorphic on $\H$ and at the cusps,
and which have Fourier coefficients in $A$. 
Denote by $S_k(m, \chi, A)$ the submodule of cusp forms and by $M_k^!(m, \chi, A)$  the module  of weakly holomorphic modular forms (which are allowed poles at the cusps).  
We omit $A$ from the notation when it is the field of algebraic numbers and $\chi$ when it is the trivial character.  By a modular function on 
 $\Gamma_0(m)$ we mean a meromorphic function which  satisfies \eqref{eq:transform} with $k=0$ and trivial $\chi$. 

For even $k \ge 2$ the weight $k$ Eisenstein series $E_k$ is defined by
\begin{gather}\label{eq:eis}
    E_k:= 1 - \frac{2k}{B_k} \sum_{n=1}^{\infty}\sigma_{k-1}(n)q^n.
\end{gather}
This is a modular form when $k \ge 4$ and a quasi-modular form when $k=2$. 
For $\ell\geq 5$ we have the congruences
\begin{gather}\label{eq:Ekmodp}
    E_{\ell+1} \equiv E_2 \pmod{\ell}, \qquad  E_{\ell-1} \equiv 1 \pmod{\ell}. 
\end{gather}

The Fricke involution on $M_k^!(m)$ is given by $f\to f\sl _kW_m$, where
\begin{gather}\label{eq:fricke_involution}
W_m = \pMatrix 0{-1}m0.
\end{gather}
This preserves the spaces of modular and cusp forms.
If $f \in M_{k}^!(m)$, then we have Fourier expansions in powers of $q:=e^{2\pi i z}$:
\begin{gather}\label{eq:ord_cusps}
    f = a(r_\infty)q^{r_\infty}+\cdots,
    \qquad  f\sl_kW_m = a(r_0)q^{r_0}+\cdots,  \qquad a(r_\infty), a(r_0)\neq 0.
\end{gather}
 These correspond to the expansion of $f$ at the cusps $\infty$, $0$ of $\Gamma_0(m)$, and we 
  define $\ord_\infty(f):=r_\infty$
 and $\ord_0(f):=r_0$.

We require some facts which follow from the $q$-expansion principle \cite[ VII, Corollary 3.12]{DeligneRapoport}  (see \cite[Lemma 2.3]{ahlgren-beckwith-raum} for a convenient formulation).
Suppose that $\ell$ is a prime with $\ell\nmid m$.  
If $\gamma=\spmatrix abcd\in \SL_2(\Z)$
with  $\ell\nmid c$ and $f, g\in M_k^!(m)$ have $\ell$-integral coefficients, then the same is true of 
 $f\sl _k\gamma$, $g\sl _k\gamma$,
 and 
 \begin{gather}\label{eq:qexp}
     f\equiv g\pmod\ell\iff f\sl _k \gamma\equiv g\sl _k\gamma\pmod\ell.
 \end{gather}
If $f\in M_k^!(m)$ has $\ell$-integral
Fourier coefficients then (still assuming that $\ell\nmid m$)  the same is true of $f\sl _kW_m$, and 
it makes sense to define 
$\ord_\infty(f\pmod\ell):= r'_\infty$
and 
$\ord_0(f\pmod\ell)= r'_0$, where (with 
$a(r'_\infty), a(r'_0)\not\equiv  0\pmod\ell$) we have 
\begin{gather}
   f \equiv a( r'_\infty)q^{ r'_\infty}+\cdots\pmod\ell,
    \qquad   f\sl_kW_m \equiv a( r'_0)q^{ r'_0}+\cdots\pmod\ell.
\end{gather}

If $f=\sum a(n)q^n$ then we define the theta operator by 
\begin{gather}\label{eq:thetadef}
\theta f:=q\frac{df}{dq}=\sum na(n)q^n.
\end{gather}
We   require some standard facts about filtrations.
 If  $\ell \geq 5$ is prime and 
  $f =\sum a(n)q^n\in \Z[\![q]\!]$  then we define
\begin{gather}
\bar f:= \sum \bar{a(n)}q^n \in \left(\Z/\ell\Z\right)[\![q]\!].
\end{gather}
If  $\bar f$ is the reduction modulo $\ell$ of an element  $M_k(m, \Z)$ for some $k$ then we define the filtration 
\begin{gather}
w (\bar f):= \inf\{ k':  \ \text{there exists $g \in M_{k'}\left(m,\Z\right)$ with $\bar f=\bar  g$}\}.
\end{gather} 
We write $\omega(f)=\omega(\bar f)$ for convenience. 

We have the following basic facts (see e.g. \cite[\S 1]{Jochnowitz}).
\begin{lemma}\label{lem:filt_theta}
Suppose that $\ell\geq 5$ is a prime with $\ell\nmid m$  and that $\bar f\neq 0$
is the reduction modulo $\ell$ of an element of $M_ k\left(m,\Z\right)$. Then we have:  
\begin{enumerate}
    \item 
If $g \in M_{k'}\left(m,\Z\right)$ has $\bar f=\bar  g$ then  $k \equiv k'\pmod{\ell-1}$.
\item 
     $\omega(\theta\, f) \leq \omega(f) + \ell +1$, with equality if and only if $\omega(f) \not \equiv 0 \pmod \ell. $
\item  $\omega(\theta\, f)\equiv \omega(f)+2 \pmod{\ell-1}$.
\end{enumerate}
\end{lemma}
We will need the following result at several points in our proofs.
\begin{lemma}\label{lem:thetaord}
    Let $\ell, m \geq 5$ be distinct primes. Suppose that $f\in M_{k_0}(m)$ vanishes to orders $r_\infty$ and  $r_0$ at  $\infty$ and $0$, and that $\omega(\theta^i\, f)=k$. 
    Then there exists $g\in M_k(m)$ with 
    $g\equiv \theta^i\, f\pmod\ell$ and 
    \begin{gather}
     \ord_\infty(g\pmod\ell)\geq r_\infty,\qquad  
         \ord_0(g \pmod\ell)\geq r_0. 
    \end{gather}
\end{lemma}
\begin{proof}
 Let  $\partial_{k_0}$ be the
       Ramanujan-Serre derivative, which is  given by 
 \begin{gather}
     \partial_{k_0} f  := \theta\, f -\frac{k_0}{12} E_2 f \in M_{k_0+2}(m).
\end{gather}
      It is straightforward to show that we have $\ord_\infty \partial_{k_0} f \geq r_\infty$ and  $ \ord_0 \partial_{k_0} f \geq r_0$.
Now let
\begin{gather}
h:= E_{\ell-1}\,\partial_{k_0} f +\frac{k_0}{12}E_{\ell+1}\, f\in M_{k_0+\ell+1}(m).    
\end{gather}
 By \eqref{eq:Ekmodp} we have  $h\equiv \theta\, f\pmod\ell$.  We also have 
\begin{gather}
     h\sl_{k_0+\ell+1} W_m  = m^{\frac{\ell-1}{2}}E_{\ell-1}(mz) \,\partial_{k_0} f\sl_{k_0+2}W_m + \frac{k_0}{12}m^{\frac{\ell+1}{2}}E_{\ell+1}(mz)\,f\sl_{k_0}W_m,
 \end{gather}
so 
 $\ord_\infty h \geq r_\infty$ and  $ \ord_0 h  \geq r_0$.
 
Iterating this process and using the assumption that $\omega\left(\theta^i\, f\right)=k$, 
we see that there  are 
    $H\in M_{k_0+i(\ell+1)}(m)$  and $g\in M_k(m)$ such that  $\ord_\infty H \geq r_\infty$,
    $\ord_0 H \geq r_0$, and
    \begin{gather}
        H\equiv g\equiv \theta^i \,f\pmod\ell.
    \end{gather}
    This gives a congruence between two modular forms of weight $k_0+i(\ell+1)$:
    \begin{gather}
        H\equiv g E_{\ell-1}^\frac{k_0+i(\ell+1)-k}{\ell-1}\pmod\ell.
    \end{gather}
From the $q$-expansion principle we obtain
\begin{gather}
    H\sl_{k_0+i(\ell+1)}W_m\equiv g\sl_kW_m\cdot m^{\frac{k_0+i(\ell+1)-k}{2}}E_{\ell-1}^\frac{k_0+i(\ell+1)-k}{\ell-1}(mz)\pmod\ell.
\end{gather}
    The lemma follows from these facts.
\end{proof}

Finally, we will frequently encounter the 
 Dedekind eta function, which  is the modular form of weight $1/2$ on $\SL_2(\Z)$ defined by 
 \begin{gather}\label{eq:etadef}
\eta(z)=q^\frac1{24}\prod_{n=1}^\infty(1-q^n).
  \end{gather}

\section{Classifying congruences for \texpdf{$\cpm$}{c phi m}}
\label{sec:classify}
In this section we consider arbitrary primes $\ell$ and $m$ with $\ell>m\geq 5$, and    we deduce  restrictive consequences of the existence of a congruence 
\begin{gather}\label{eq:genericcong}\cpm(\ell n+\beta)\equiv 0\pmod \ell\qquad\text{for all $n$.}     
\end{gather}
Chan, Wang, and Yang \cite[Thm.~$2.1$]{Chan-Wang-Yang} proved that if $m$ is a positive odd integer, then 
\begin{equation}\label{eq:Amdef}
 A_m(z):=\prod_{n \geq 1}(1-q^n)^m\sum^\infty_{n=0} \cpm(n)q^n \in M_\frac{m-1}2\left(m, \ptfrac\bullet m\right).
\end{equation}
Fix a prime $m\geq 5$, and for 
 each prime $\ell>m$  define
\begin{gather} \label{eq:hldef} h_\ell(z):=\frac{\eta^{\ell^2}(mz)}{\eta^m(z)}A_m(z)
=\eta^{\ell^2}(mz)\sum \cpm\pfrac{ n+m}{24}q^\frac n{24}
=q^\frac{(\ell^2-1)m}{24}+\cdots.
\end{gather}  
 By a standard 
 criterion for eta-quotients \cite[Thm. 1.64]{ono_web})
we have 
$h_\ell\in S_{(\ell^2-1)/2}(m,\Z)$.  Since $\eta^{\ell^2}(mz)\equiv \eta(\ell^2 mz)\pmod\ell$  we also  have   
\begin{gather}\label{eq:h_ell_mod_ell}
       h_{\ell}(z) \equiv q^{\frac{(\ell^2-1)m}{24}}\sum_{n=0}^{\ell^2m-1}\cpm(n)q^n +O\left(q^{\ell^2m+\frac{(\ell^2-1)m}{24}}\right) \pmod \ell.
\end{gather}
The following   generalization of a result of Kiming and Olsson \cite{KimingOlsson} reduces us to considering a single residue class modulo $\ell$.
\begin{proposition}\label{prop:kimols}
Suppose that $\ell > m \ge 5$ are primes and  that for some $\beta\in \Z$ we have
\begin{gather}\label{eq:cpmcong0}
    \cpm(\ell n+\beta)\equiv 0\pmod \ell\qquad\text{for all $n$.} 
\end{gather}
 Then $\beta\equiv m/24\pmod\ell$.  
 
 In other words the congruence \eqref{eq:cpmcong0} can occur only if
\begin{gather}\label{eq:cpmcong}
h_\ell\sl U_\ell\equiv 0\pmod\ell.
 \end{gather}
\end{proposition}
Once this is proved the next proposition will provide a main tool  used in the following sections to classify the existence of congruences for particular values of $m$.
\begin{proposition}\label{prop:lowpoint}
   Suppose that $\ell > m\geq 5$ are prime. 
   If $h_\ell\|U_\ell\equiv 0\pmod\ell$ then 
   \begin{gather}\label{eq:lowpoint}
       w\left(\theta^\frac{\ell+3}2\,h_\ell\right)=\frac{\ell^2-1}2+4.
   \end{gather}
\end{proposition}

The rest of the section is devoted to the proof of these two propositions.
We begin by determining the filtration of the modular form $h_\ell$.

\begin{lemma}\label{lem:filt_h_l}  If $m>\ell\geq 5$ are primes then 
\begin{gather}
\omega(h_{\ell}) = \frac{\ell^2-1}{2}.
\end{gather}  
\end{lemma}
\begin{proof}
The modular form $A_m$ does not vanish at the cusp zero by \cite[(4.10)]{Chan-Wang-Yang}.
From this fact together with  \eqref{eq:hldef}, we find that
 \begin{gather} \label{eq:ord_hl}
r_\infty:=\ord_{\infty}(h_\ell) = \frac{(\ell^2-1)m}{24}\ \ \  \text{ and } \ \ \ r_0:=\ord_{0}(h_\ell) = \frac{\ell^2-m^2}{24}.
 \end{gather}
Assume by way of contradiction that 
there exists a modular form  $g_{\ell} \in M
_{(\ell^2-1)/2-(\ell-1)}(m)$ 
with  $g_\ell\equiv h_\ell\pmod\ell$. 
We then consider the multiplicative trace of $g_{\ell}$:
\begin{gather}\label{eq:mult_trace}
    H:= \prod_{\gamma \in \Gamma_0(m) \backslash \Gamma_0(1)} g_{\ell}\sl_{\frac{\ell^2-1}2-(\ell-1)}\gamma 
    \in M_{(m+1)\left(\frac{\ell^2-1}{2}- (\ell-1)\right)}(1).
\end{gather}
For each coset representative $\gamma$, 
\eqref{eq:qexp} gives
\begin{gather} h_\ell\sl_\frac{\ell^2-1}2 \gamma\equiv  \left (g_\ell E_{\ell-1}\right)\sl_\frac{\ell^2-1}2 \gamma
\equiv g_\ell\sl_{\frac{\ell^2-1}2-(\ell-1)} \gamma\cdot E_{\ell-1}\pmod\ell.
\end{gather}
Since $m$ is prime, the two nonequivalent cusps of  $\Gamma_0(m)$ are $\infty$ and $0$, and each non-identity coset $\gamma$ maps $\infty$ to $0$.
Using \eqref{eq:ord_hl} and \cite[Proposition III.16]{Koblitz} we see that for each such $\gamma$ there is a non-zero constant $c_\gamma$ such that 
\begin{gather}\label{eq:non_ident_coset}
  g_\ell\sl_{\frac{\ell^2-1}2-(\ell-1)} \gamma
  =c_\gamma q^\frac{r_0}m+\cdots.
\end{gather} 
Since there are $m$ non-identity cosets, we obtain
\begin{gather}\label{eq:ordF}
\ord_{\infty}(H \pmod{\ell}) = \frac{(\ell^2-1)m}{24} + \frac{\ell^2-m^2}{24}.
\end{gather}    
From the Sturm bound, we must have  
\begin{gather}  \frac{(\ell^2-1)m}{24} + \frac{\ell^2-m^2}{24}\leq \frac{m+1}{12}\left(\frac{\ell^2-1}{2}-(\ell-1)\right).\end{gather}
The lemma follows since  this contradicts the assumption that $\ell >m$.
\end{proof}

\begin{proof}[Proof of Proposition~\ref{prop:kimols}]
We adapt parts of the proof of \cite[Theorem 3]{KimingOlsson}.
Suppose that we have a congruence 
\begin{gather}\label{eq:betacong}
\cpm(\ell n+\beta)\equiv 0\pmod \ell,
\end{gather}
and define
\begin{gather}\label{eq:a_def}
a:=\beta+\frac{(\ell^2-1)m}{24}.
\end{gather}
From \eqref{eq:hldef} we find that 
\begin{gather}\label{eq:hl_cong}
    q^{-a}h_\ell\equiv \prod_{n=1}^\infty\left(1-q^{\ell^2m n}\right)\sum \cpm(n+\beta)q^n\pmod\ell,
    \end{gather}
from which we see that 
\begin{gather}
    \left(q^{-a}h_\ell \right)\sl U_\ell\equiv 0\pmod\ell.
\end{gather}
Since for any $f=\sum a(n)q^n\in \Z[\![q]\!]$ we have 
\begin{gather}
    \left(f\sl U_\ell\right)^\ell\equiv f-\theta^{\ell-1}\, f \pmod\ell, 
\end{gather}
we conclude that the congruence \eqref{eq:betacong} leads to 
\begin{gather}\label{eq:kimolsProp1}
     \theta^{\ell-1}\,(q^{-a}h_{\ell}) \equiv q^{-a}h_{\ell} \pmod \ell.
\end{gather}

Now suppose by way of contradiction that  $a\not\equiv  0\pmod\ell$.
By the product rule we have 
\begin{gather}
    \theta^{\ell-1}\,(q^{-a}h_\ell) \equiv \sum_{i =0}^{\ell-1} \binom{\ell-1}{i}\theta^{\ell-1-i}(q^{-a})\theta^{i}\, h_\ell  \equiv q^{-a}\sum_{i =0}^{\ell-1} a^{\ell-1-i}\theta^{i}\, h_\ell \pmod \ell. 
\end{gather}
From \eqref{eq:kimolsProp1} this gives 
\begin{gather}\label{eq:prod_rule_simp}
    \sum_{i=1}^{\ell-1}a^{\ell-1-i}\theta^i \,h_{\ell} \equiv 0 \pmod \ell.
\end{gather}
By Lemma~\ref{lem:filt_theta} 
there are only two terms in the sum whose filtration is congruent to $2$ modulo $\ell-1$:
\begin{gather}\label{eq:thetahlFil}
\begin{split}
   \omega(\theta \,h_{\ell})\equiv \omega(\theta^{\frac{\ell+1}{2}}\,h_{\ell}) \equiv 2\pmod{\ell -1}.
\end{split}  
\end{gather}
Since the module of modular forms of level $m$ modulo $\ell$ is graded by $\Z/(\ell-1)\Z$ 
\cite[p. 459]{Gross}, we must have 
\begin{gather}
\omega(\theta \,h_{\ell}) = \omega\left(\theta^{\frac{\ell+1}{2}}\, h_{\ell}\right).
\end{gather}
This is impossible since 
by Lemmas~\ref{lem:filt_theta} and \ref{lem:filt_h_l}
we have 
\begin{gather}
  \omega(\theta^j \,h_{\ell}) =  \frac{\ell^2-1}2+j(l+1),\qquad 0\leq j\leq \frac{\ell+1}{2}.
\end{gather}
Therefore $a\equiv\beta-m/24\equiv 0\pmod\ell$ and the first assertion of the proposition is proved.  The second is then immediate from \eqref{eq:hl_cong}.
\end{proof}

To prove Proposition~\ref{prop:lowpoint}, we require a lower bound for the filtration of each element in the theta cycle for $h_{\ell}$. 
\begin{lemma}\label{lem:repeat_theta_hl} If $\ell > m \geq 5$ are primes and $i \in \mathbb{N}$ then 
\begin{gather}\omega(\theta^i \, h_{\ell}) \geq \frac{\ell^2 - m}{2}.\end{gather}
\end{lemma}
\begin{proof}
    Suppose that $\omega(\theta^i\, h_{\ell}) = k$, and recall the definition \eqref{eq:ord_hl}.
    By Lemma~\ref{lem:thetaord} there is a modular form 
    $\tilde h\in S_{k}(m)$ such that $\tilde{h}\equiv \theta^i\, h_\ell\pmod\ell$ and such that $\tilde h$
    vanishes at the cusps to orders at least $r_\infty$ and $r_0$.
    Define 
    \begin{gather}
    \tilde{H} :=  \prod_{\gamma \in \Gamma_0(m) \backslash \Gamma_0(1)} \tilde h\sl_k\gamma \in S_{(m+1)k}(1).   
    \end{gather}
As in the proof of Lemma~\ref{lem:filt_h_l}  we find that $\tilde H$ is $\ell$-integral and that 
    \begin{align*}
        \ord_{\infty}(\tilde{H}\pmod \ell) \geq 
        r_\infty+r_0=\frac{(\ell^2-1)m}{24}+\frac{\ell^2-m^2}{24}.
    \end{align*}
From the Sturm bound we must have 
    \begin{gather}
    \frac{(\ell^2-1)m}{24}+\frac{\ell^2-m^2}{24} \leq \frac{(m+1)k}{12},
    \end{gather} and 
this inequality reduces to 
        $(\ell^2 - m)/2 \leq k$, 
    which proves the lemma.
\end{proof}

\begin{proof}[Proof of Proposition~\ref{prop:lowpoint}]

Recall from Lemma~\ref{lem:filt_h_l} that  
\begin{gather}\label{eq:filtremind}
\omega (h_\ell)=\frac{\ell^2-1}2;
\end{gather} it follows from Lemma~\ref{lem:filt_theta} that $\omega\left(\theta^\frac{\ell+1}2\,h_\ell\right)\equiv 0\pmod\ell$ and that for some positive integer $\alpha$ we have
\begin{gather}\label{eq:filt_theta_l+3/2}
\omega\left(\theta^\frac{\ell+3}2\,h_{\ell}\right) = \frac{\ell^2-1}2+\frac{\ell+3}2(\ell+1) -\alpha(\ell-1).
\end{gather}
By Lemma~\ref{lem:repeat_theta_hl} and the fact that $\ell\geq m+2$ we conclude that 
\begin{gather}
\alpha\leq \frac{\ell+5}2+\frac12+\frac3{\ell-1},    
\end{gather}
from which it follows that 
\begin{gather}\alpha\leq \frac{\ell+5}2\ \ \ \text{when} \ \ \ \ell\geq 11.
\end{gather}
We will prove that 
\begin{gather}\label{eq:alphaval}\alpha= \frac{\ell+5}2,
\end{gather}
 from which the proposition follows immediately by \eqref{eq:filt_theta_l+3/2}.
 
When $\ell=7$ we need only consider $m=5$, and we can 
check \eqref{eq:alphaval} directly.
For $\ell\geq 11$ assume by way of contradiction that $\alpha\leq(\ell+3)/2$.
Then \eqref{eq:filt_theta_l+3/2} gives
\begin{gather}\label{eq:thetafacts}
\omega\left(\theta^\frac{\ell+3}2\,h_{\ell}\right)\geq\frac{\ell^2-1}2+\ell+3,\qquad \omega\left(\theta^\frac{\ell+3}2\,h_{\ell}\right)\equiv 1+\alpha\pmod\ell.
\end{gather}
By Lemma~\ref{lem:filt_theta} we conclude that 
\begin{gather}\label{eq:bigfilt}
\omega\left(\theta^{\frac{\ell+3}2+j}\,h_{\ell}\right)\geq 
\frac{\ell^2-1}2+\ell+3+j(\ell+1),\ \ \ \ 
1\leq j\leq \ell-\alpha-1.
\end{gather}
Setting $j=(\ell-5)/2$ in \eqref{eq:filt_theta_l+3/2} provides a   contradiction to 
\eqref{eq:filtremind} since by assumption we have 
$\theta^{\ell-1}\, h_\ell\equiv h_\ell\pmod \ell$.
\end{proof}

\section{\texorpdfstring{The case $m=5$}{m=5}}
To prove the main theorem in the case when $m =5$
we will require  explicit bases for spaces of modular forms of level 5.
For  $k\equiv 0\pmod 4$ 
we define 
\begin{gather}\label{eq:b5def}
F_r(z):=\eta^{k}(5z)\eta^{k}(z)\cdot\left(\frac{\eta^6(5z)}{\eta^6(z)}\right)^{r-\frac k4}
=q^r+\cdots,\qquad 1\leq r\leq \frac k2-1.
\end{gather}
We  see that $F_r$ has order $k/2-r$ at zero.
By \cite[Thm. 1.64]{ono_web})
we have 
$F_r\in S_k(5)$, 
and by  standard dimension formulas \cite[Theorem~3.5.1]{Diamond-Shurman} we have 
$\dim\left(S_k(5)\right)=k/2-1$. 
It follows that the $F_r$ 
form a basis for $S_k(5)$.

Assuming the existence of a congruence, we next show that the form at the low point of the theta cycle of $h_\ell$ can be written as a linear combination of only four of these basis elements.
In this case we have 
\begin{gather}
     r_\infty=\frac{5(\ell^2-1)}{24},\qquad 
     r_0=\frac{\ell^2-25}{24}. 
\end{gather}

\begin{lemma}\label{lem:lin_combo_5} Suppose that $m=5$ and that  $\ell \geq 7$ is prime. If $h_{\ell}\sl U_{\ell} \equiv 0 \pmod \ell$ then there exist integers $\alpha_{i, \ell}$  such that
\begin{gather}\label{eq:lin_combo_finite5}
    576\cdot \theta^{\frac{\ell+3}{2}}\,h_{\ell} \equiv \sum_{i =0}^{3} \alpha_{i, \ell}F_{r_{\infty}+i} \pmod \ell,
\end{gather}
where $\{F_r\}$ is the basis in \eqref{eq:b5def} with $k=(\ell^2-1)/2+4$. Moreover, we have 
\begin{gather}
\label{eq:F_r_m5_expan}
\begin{aligned}
    \frac{F_{r_{\infty}}}{q^{r_{\infty}}} &= 1-10q+35q^2-30q^3-105q^4 + (241-\ell^2)q^5+ \cdots\\
\frac{F_{r_{\infty}+1}}{q^{r_{\infty}}} &=   q-4q^2+2q^3+8q^4 -5q^5+ \cdots\\
\frac{F_{r_{\infty}+2}}{q^{r_{\infty}}} &= q^2+2q^3+5q^4 +10q^5+ \cdots\\    \frac{F_{r_{\infty}+3}}{q^{r_\infty}} &=q^3 + 8q^4+44q^5 + \cdots.\\
\end{aligned}
\end{gather}

\end{lemma}
\begin{proof}
The expansions  \eqref{eq:F_r_m5_expan} can be computed directly.  For the other statement
    we use Proposition~\ref{prop:lowpoint} and \eqref{eq:b5def} to find that for some $\alpha_r\in \Z$ we have
 \begin{gather}\label{eq:lin_combo_all_5}
     576\cdot \theta^{\frac{\ell+3}{2}}\,h_{\ell} \equiv \sum_{r=1}^{(\ell^2+3)/4}\alpha_r F_r \pmod \ell.
 \end{gather}
 By Lemma~\ref{lem:thetaord}, there exists  $g_{\ell}\in S_{(\ell^2-1)/2+4}(5)$ such that $g_{\ell} \equiv \theta^{\frac{\ell+3}{2}}\, h_{\ell} \pmod \ell$ 
 and such that 
 \begin{gather}
\ord_\infty\left(g_\ell\pmod\ell\right)\geq r_\infty,\ \
\ 
\ord_0\left(g_\ell\pmod\ell\right)\geq r_0. 
 \end{gather}
We conclude   that 
\begin{gather}
    g_\ell\equiv \sum_{r=r_\infty}^{r_\infty+3}\alpha_r F_r\pmod\ell;
\end{gather}
 the lower limit is clear from \eqref{eq:b5def} 
 and  the upper limit 
 follows from  the discussion after \eqref{eq:qexp}, the fact that $F_r$ has order $(\ell^2+7)/4-r$ at zero, and the fact that 
  $(\ell^2+7)/4-(r_\infty+3)=r_0$.
\end{proof}

\begin{proof}[Proof of Main Theorem when $m =5$]

Let $\ell \geq 7$ be prime. 
    Computing from \eqref{eq:h_ell_mod_ell}, we find that 
\begin{gather}\label{eq:thetalowhl}
    576\cdot \frac{\theta^{\frac{\ell+3}{2}}\,h_{\ell}(z)}{q^{r_\infty}} \equiv 
    \sum_{n=0}^{5}c\phi_5(n)(24n-5)^2\left(\frac{24(24n-5)}{
    \ell}\right)q^n +O(q^6) \pmod \ell.
\end{gather}
Defining \begin{gather}
   \underline b:=\left(c\phi_5(n)(24n-5)^2\right)_{0\leq n\leq 5}=(25,9025, 277350,  3030075, 20288450, 104160100)
\end{gather}
and
 \begin{gather}
\underline\ep_\ell:=\left(\left(\frac{24(24n-5)}{ \ell}\right)\right)_{0\leq n\leq 5}\in \{0, \pm1\}^6,
\end{gather}
we see that the Hadamard product $\underline b\odot\underline\ep_\ell$ gives the first six coefficients of the function in \eqref{eq:thetalowhl} modulo $\ell$.

 Fix one of the finitely many choices of $\underline\ep=(\ep_0, \dots, \ep_{5})\in \{0, \pm1\}^{6}$. 
Let $E(\underline\ep)$
be the set of $\ell$ such 
that $h_\ell\sl U_\ell\equiv 0\pmod\ell$ and $\underline\ep_\ell=\underline\ep$.
Note that each  $\underline \ep_\ell$ has at most one zero entry because there is at most one solution to $24n -5\equiv 0 \pmod \ell$ for $ 0 \leq n \leq 5$.  So we need only   consider the 256 choices of 
 $\underline\ep$ with at most one zero entry.

    We will show that $E(\underline\ep)$ is empty 
apart from the two values of $\underline\ep$
which lead to the congruences
 \eqref{eq:ram-cong5}.  Using  the matrix of coefficients of the $F_{r_\infty+i}/q^{r_\infty}$ in \eqref{eq:F_r_m5_expan},  we compute the  integer vector $(\alpha_0, \alpha_1, \alpha_2, \alpha_3)^T$ such that
    \begin{gather} \label{eq:sys1_5}
   \begin{pmatrix}
             1 & 0 & 0 & 0 \\
 -10 & 1 & 0 & 0 \\
 35 & -4 & 1 & 0 \\
 -30 & 2 & 2 & 1 \\
        \end{pmatrix}
\begin{pmatrix}
    \alpha_0\\
    \alpha_1\\
    \alpha_2\\
    \alpha_3
\end{pmatrix} = 
\begin{pmatrix}
    25\,\ep_0\\
    9025\,\ep_1\\
    277350\, \ep_2\\
    3030075\, \ep_3
\end{pmatrix} .
    \end{gather}
By Lemma~\ref{lem:lin_combo_5}, we see that
\begin{gather} \label{eq:sys2_5}
\begin{pmatrix}
     -105 & 8 & 5 & 8 \\
 241 & -5 & 10 & 44
\end{pmatrix}  \begin{pmatrix}
    \alpha_0\\
    \alpha_1\\
    \alpha_2\\
    \alpha_3
\end{pmatrix}   \equiv  \begin{pmatrix}
           20288450\,\ep_4\\
           104160100\,\ep_5 
          \end{pmatrix}
          \pmod \ell.
\end{gather}
For each   $\underline\ep$, 
we compute the finite set $S(\underline\ep)$ of primes $\ell$ for which 
\eqref{eq:sys2_5} has solutions.
For $\ell\in S(\underline\ep)$ we then determine that $\underline\ep_\ell\neq \underline\ep$, and so that $E(\underline\ep)=\emptyset$, with only two exceptions:

\begin{gather}\label{eq:eps_7}
    \ell=7, \qquad 
    \underline \ep=\underline \ep_7 = (-1,1,-1,-1,0,1),
    \qquad E(\underline \ep_7)=\{7\},
\end{gather}
\begin{gather}\label{eq:eps_11}
        \ell=11, \qquad 
\underline \ep=\underline \ep_{11} =(1,1,1,-1,-1,-1),
\qquad E(\underline \ep_{11})=\{11\}.
\end{gather} 
We provide an example of these computations below.
This shows that there are no primes 
$\ell \geq 13$ with $h_{\ell}\sl U_{\ell} \equiv 0 \pmod \ell$  and completes the proof when $m=5$.
\end{proof}

\subsection{Examples}
We describe the computations corresponding to the congruences \eqref{eq:ram-cong5}. When $\underline \ep = \ep_7$, we see that \eqref{eq:sys1_5} becomes 
\begin{gather}
\begin{pmatrix}
             1 & 0 & 0 & 0 \\
 -10 & 1 & 0 & 0 \\
 35 & -4 & 1 & 0 \\
 -30 & 2 & 2 & 1 \\
        \end{pmatrix}
 \begin{pmatrix}
    -25\\
    8775\\
    -241375\\
    -2565625
\end{pmatrix} = \begin{pmatrix}
    -25\\
    9025\\
    -277350 \\
    -3030075
\end{pmatrix}, 
    \end{gather}
and that  \eqref{eq:sys2_5} becomes
\begin{gather}
    \begin{pmatrix}
     -105 & 8 & 5 & 8 \\
 241 & -5 & 10 & 44
\end{pmatrix}  \begin{pmatrix}
    -25\\
    8775\\
    -241375\\
    -2565625
\end{pmatrix}
=
\begin{pmatrix}
    -21659050 \\
 -115351150 \\
\end{pmatrix}
\equiv 
  \begin{pmatrix}
           0\\
          104160100 
        \end{pmatrix}
          \pmod \ell.
\end{gather}
We find that 
 $S(\underline\ep_7)$ is  the set of primes $\ell$ 
 which divide 
\begin{gather}
\gcd(21659050,104160100+115351150)=350=2\cdot 5^2\cdot7,
\end{gather}
 from which $E(\underline\ep_7)=\{7\}$.

 Similarly,  when $\underline \ep = \underline \ep_{11}$, \eqref{eq:sys1_5} becomes 
\begin{gather}
\begin{pmatrix}
     1 & 0 & 0 & 0 \\
 -10 & 1 & 0 & 0 \\
 35 & -4 & 1 & 0 \\
 -30 & 2 & 2 & 1 \\
\end{pmatrix}
 \begin{pmatrix}
    25\\
    9275\\
    313575\\
    -3675025
\end{pmatrix} = \begin{pmatrix}
    25\\
    9025\\
    277350\\
    -3030075
\end{pmatrix},
    \end{gather}
and we search for the primes $\ell$ satisfying 
\begin{gather}
    \begin{pmatrix}
     -105 & 8 & 5 & 8 \\
 241 & -5 & 10 & 44
\end{pmatrix}  \begin{pmatrix}
    25\\
    9275\\
    313575\\
    -3675025
\end{pmatrix} 
=\begin{pmatrix}
    -27760750\\
    -158605700
\end{pmatrix}
\equiv 
  \begin{pmatrix}
           -20288450\\
           -104160100
          \end{pmatrix}
          \pmod \ell.
\end{gather}
This occurs  only if $\ell\mid 1100$,  from which we conclude that 
$E(\underline\ep_{11})=\{11\}$.

\section{\texorpdfstring{The case $m=11$}{m=11}}
Although the basic method is the same, the situation here is significantly more  complicated.
A Mathematica notebook with the computations described in this section is available at 
\cite{githubfile}.

\subsection{Construction of bases}
By  \cite[Theorem 3.5.1]{Diamond-Shurman}
we have 
\begin{gather}\label{eq:dim11}
    \dim S_k(11)=k-2\ \ \ \ \text{if $k\geq 4$}.
\end{gather}
Here we outline a construction of explicit bases for $S_k(11)$ which are suited to the computations of the next section; since the construction is somewhat technical we introduce some notation.
For $x\in \{0, 1, 2, 3, 4\}$ we  define 
\begin{gather}
    \label{eq:alphadef}
    \alpha(x):=\begin{cases}
    0&\quad\text{if \ $x=0$,}\\
    2&\quad\text{if \ $x=1$,}\\
    3&\quad\text{if \ $x=2$,}\\
    4&\quad\text{if \ $x=3$,}\\
    6&\quad\text{if \ $x=4$.}\\
    \end{cases}
\end{gather}
and 
\begin{gather}\label{eq:xidef}
    \xi(x):=\alpha(x)-x\in \{0, 1, 2\}.
\end{gather}
For integers $r, k$ we also define 
$\gamma_{r, k}\in \{0, 1, 2, 3, 4\}$ by 
\begin{gather}\label{eq:gammadef}    \gamma(r,k):=2k+r\pmod 5.
\end{gather}
We will prove 
\begin{proposition}\label{prop:basis11}
 Suppose that $k\geq 4$.  
 For $r\geq 1$, define the modular forms $F_r$  by \eqref{eq:frdef}.
 Then a basis for $S_k(11)$ is given by
\begin{gather}
    \begin{aligned}
        &\{F_1, \cdots F_{k-2}\}\quad&\text{if $k\not\equiv 2\pmod 5$,}\\
        &\{F_1, \cdots F_{k-3},  F_{k-1}\}\quad&\text{if $k\equiv 2\pmod 5$.}\\
    \end{aligned}
\end{gather}
Furthermore:
\begin{enumerate}
    \item $F_r$ has integer coefficients,
    \item $F_r$ has order $r$ at infinity,
    \item \label{itm: order_zero} $F_r$ has order $k-r-\xi(\gamma(r,k))$ at zero,
    \item The various orders at zero in (\ref{itm: order_zero})  are all distinct.
\end{enumerate}
\end{proposition}
We begin by constructing a family of modular functions which is crucial to the proof.
\begin{lemma}\label{lem:mod_fct_11}
    For $x\in \{0, 1, 2, 3, 4\}$ the function $h_x$ defined in \eqref{eq:hxdef} is a modular function on $\Gamma_0(11)$ which is holomorphic on the upper half-plane and  has integer coefficients, expansion $h_x=q^{-x}+\cdots$ at $\infty$, and a pole of order $\alpha(x)$ at zero.
\end{lemma}
\begin{proof} We begin by defining two weight two modular forms on $\Gamma_0(11)$:
\begin{align}
g(z)&:=\left(\sum_{x, y\in \Z}q^{x^2+xy+3y^2}\right)^2=1+2q+4q^3+\cdots,\\
h(z)&:=\eta^2(11z)\eta^2(z)=q-2q^2-q^3+\cdots.   
\end{align}
We then define the modular functions
\begin{gather}
    F:=\frac {g^2} h=q^{-1}+6 +17q+\cdots,\qquad G:=\frac{\theta\, F}h=-q^{-2} -2q^{-1} + 12+\cdots.
\end{gather}
These functions (both of which are holomorphic on the upper half-plane) were used in \cite{weston} to determine a model for the elliptic curve $X_0(11)$.
We have 
\begin{gather}\label{eq:F_G_W11}
    F\sl W_{11}=F,\qquad G\sl W_{11}=-G,
\end{gather}
and using this fact we can construct the modular functions $h_x$ with prescribed pole orders at $0$ and infinity by constructing appropriate polynomials in $F$ and $G$.
In particular we define  
\begin{gather}
    \label{eq:hxdef}
\begin{aligned}
    h_0&:=1,\\
    h_1&:=-\tfrac1{121} \left(-\tfrac12 F^2 + 5 F - \tfrac12 G + 11\right)=q+5q^2+\cdots,\\
    h_2&:=-\tfrac1{2662} \left(-88 F + 21 F^2 - F^3 + 11 G - F G\right)=q^2+9q^3\cdots,\\
    h_3&:=\tfrac1{29282} \left(-11 F^3 + F^2 G + 231 F^2 - 21 F G - 726 F + G^2 + 99 G + 
   242\right)\\ 
&\ =q^3+14q^4+\cdots,\\
   h_4&:=h_1h_3=q^4+19q^5+\cdots.
\end{aligned}
\end{gather}
Computations using \eqref{eq:F_G_W11} show that each $h_x$ has order $\alpha(x)$ at zero.  For example, we have
\begin{gather}
    h_2\sl W_{11}=-\tfrac1{2662} \left(-88 F + 21 F^2 - F^3 - 11 G + F G\right)=\frac1{11^3}\left(q^{-3} - 3q^{-2} - 5q^{-1}+\cdots\right).
\end{gather}
The claim that each $h_x$ has integer coefficients can also be verified by computation: the product of $h_x$ and a suitable power of $h$ is a holomorphic modular form of level 11 which can be expressed as an integral linear combination of basis elements with integer coefficients.
\end{proof}

\begin{proof}[Proof of Proposition~\ref{prop:basis11}]
Let $k\geq 4$.  For $r\geq 1$ define 
\begin{gather}\label{eq:frdef}
F_r(z):=\eta^\frac{12r-12\gamma(r,k)-k}5(11z)
            \eta^{2k-\frac{12r-12\gamma(r,k)-k}5}(z) h_{\gamma(r, k)}.
\end{gather} 
It follows from Lemma~\ref{lem:mod_fct_11}  and \cite[Thm. 1.64]{ono_web}) that each $F_r$ transforms correctly in weight $k$ on $\Gamma_0(11)$.  We can also compute directly that 
\begin{gather}\label{eq:Fr_order}
F_r=q^r+\cdots, \qquad F_r\sl W_{11}=q^{k-r-\xi(\gamma(r, k))}+\cdots.
\end{gather}
This establishes assertions (2) and (3), and assertion (1) is clear from Lemma~\ref{lem:mod_fct_11}.

Since $\xi(\gamma(r, k))\in \{0, 1, 2\}$,    we always have
$F_1, \dots, F_{k-3}\in S_k(11)$ from \eqref{eq:Fr_order}.
We also see that 
\begin{gather}
    F_{k-2}\in S_k(11)\iff \xi(\gamma(k-2, k))\neq 2\iff \gamma(k-2, k)\neq 4\iff k\not\equiv2\pmod 5.
\end{gather}
In the case $k\equiv 2\pmod 5$ we have $\gamma(k-1, k)\equiv 0\pmod 5$, so that $\xi(\gamma(k-1, k))=0$.  Thus $F_{k-1}\in S_k(11)$ in this case.  In view of \eqref{eq:dim11} this establishes all of the assertions except for (4).
To prove (4) we assume by way of contradiction that 
\begin{gather}\label{eq:zerocont}
    k-r-\xi(\gamma(r, k))=k-r_1-\xi(\gamma(r_1, k))\ \ \ \text{where $r>r_1$}.
\end{gather}
It follows from \eqref{eq:xidef} that 
\begin{gather}
    0<r-r_1=\xi(\gamma(r_1, k))-\xi(\gamma(r, k))\in \{1, 2\}.
\end{gather}
We can check directly that if  $x, y\in \{0, 1, 2, 3, 4\}$ and
$y\equiv x+1\pmod 5$, then we never have $\xi(x)-\xi(y)=1$.
Similarly, if  $x, y\in \{0, 1, 2, 3, 4\}$ have 
$y\equiv x+2\pmod 5$, then we  have $\xi(x)-\xi(y)\neq 2$.
This shows that \eqref{eq:zerocont} does not occur, which finishes the proof of the proposition.
\end{proof}

\subsection{\texorpdfstring{Proof in the case $m=11$}{proof11}}

Here we have  $r_\infty=11(\ell^2-1)/24$ and $r_0=(
\ell^2-11^2)/24$.  Given the existence of a congruence, we show that the form at the low point of the theta cycle of $h_\ell$ is a linear combination of $9$ of the elements of the basis constructed above.
\begin{lemma}\label{lem:lin_combo}
Suppose that $\ell \geq 13$ is prime and $m =11$. If $h_{\ell}\sl U_{\ell} \equiv 0 \pmod \ell$ then there exist integers $\alpha_{i, \ell}$  such that
\begin{gather}\label{eq:lin_combo_finite11}
    576\cdot \theta^{\frac{\ell+3}{2}}\,h_{\ell} \equiv \sum_{i =0}^{8} \alpha_{i, \ell}F_{r_{\infty}+i} \pmod \ell,
\end{gather}
where $\{F_r\}$ is the basis from Proposition~\ref{prop:basis11} with $k = (\ell^2-1)/2+4$. For $0\leq i\leq 8$, write 
 \begin{gather*}
        \frac{F_{r_\infty+i}}{q^{r_\infty}}=\sum_{n=0}^\infty  c_i(n)q^n=q^i+c_{i}(i+1)q^{i+1}+\cdots.
\end{gather*}
Then the integers $c_n(i)$, $0\leq i\leq 8$, $0\leq n\leq 10$ are independent of the prime $\ell$.
\end{lemma}
\begin{proof}
The proof of \eqref{eq:lin_combo_finite11} follows from Proposition~\ref{prop:lowpoint}, Proposition~\ref{prop:basis11}, and a consideration of   the orders at zero and infinity of the cusp form $g \equiv \theta^{\frac{\ell+3}{2}}\,h_{\ell}\pmod\ell$ given by Lemma~\ref{lem:thetaord} with $k = (\ell^2-1)/2+4$.
The lower bound in the sum is clear. 
For the upper bound, Proposition~5.1 shows that the largest index $i$ which appears is the largest $i$
for which 
\begin{gather}
    k-(r_\infty+i)-\xi\left(\gamma(r_\infty+i, k)\right)\geq r_0.
\end{gather}
Since $k-r_\infty -r_0 =9$, this is the largest index $i$ with 
 $i + \xi\left(\gamma(r_\infty+i, k)\right) \leq 9$, and a computation shows that this is $i=8$.

To prove the claim about the integers $c_n(i)$ 
we begin by 
noting that $\gamma_i\equiv 3+i\pmod 5$ is independent of $\ell$. 
Define 
\begin{gather}
\beta_i:=\frac{12i-12\gamma_i-4}5\in \Z.
\end{gather}
Then a computation using \eqref{eq:frdef}
shows that 
\begin{gather} \frac{F_{r_\infty+i}}{q^{r_\infty}}  =\prod_{n=1}^\infty\left(1-q^{11}\right)^{\ell^2n}\cdot q^\frac{11}{24}\cdot\eta^{\beta_i-
1}(11z)
    \eta^{8-\beta_i}(z)\, h_{\gamma_i}=q^\frac{11}{24}\cdot \eta^{\beta_i-1}(11z)
    \eta^{8-\beta_i}(z)\, h_{\gamma_i}+O(q^{11}),
\end{gather}
which proves that the first $11$ coefficients are independent of $\ell$ and establishes the lemma.
  \end{proof}
\begin{proof}[Proof of Main Theorem when 
$m = 11$]
Let $\ell \geq 13$.
By a computation using \eqref{eq:h_ell_mod_ell}
 we see that the Hadamard product of the vectors 
\begin{gather}
    \underline b:= \left(c\phi_{11}(n)(24n-11)^2\right)_{0 \leq n \leq 10},\qquad  \underline \ep_{\ell} := \left(\left(\frac{24(24n-11)}{\ell}\right)\right)_{0\leq n \leq 10} \in \{0, \pm 1\}^{11}
\end{gather}
gives the first eleven coefficients of $576\cdot  \theta^{\frac{\ell+3}2}\, h_{\ell}$ modulo $\ell$.

 Fix one of the finitely many choices of $\underline \ep = (\ep_0,\dots, \ep_{10}) \in \{0,\pm 1\}^{10}$. Let $E(\underline \ep)$ be the set of $\ell$ such that $h_{\ell}\sl U_{\ell}\equiv 0 \pmod \ell$ and $\underline \ep_{\ell} = \underline \ep$. As in the last case,  $\underline\ep_\ell$ can have at most one zero entry. Therefore there  are at most 13,312 choices for $\underline \ep$, and we  show that $E(\underline \ep)$ is empty for each of these choices. 
 
 Using the matrix with entries $c_i(n)$ from Lemma~\ref{lem:lin_combo}, we find the integer vector $(\alpha_0,\dots, \alpha_8)^T$ such that
\begin{gather}\label{eq:system1}
\begin{pmatrix}
        1 & 0 & \cdots &0\\
        c_0(1) & \ddots  &\ddots & \vdots\\
        \vdots & \ddots & \ddots & 0 \\
        c_0(8) & \cdots & c_7(8) &1
\end{pmatrix} \begin{pmatrix}
       \alpha_0\\
      \vdots\\
      \alpha_8
 \end{pmatrix}=
 \begin{pmatrix}
       11^2 c\phi_{11}(0) \ep_0\\
       \vdots\\
      181^2 c\phi_{11}(8) \ep_8 
 \end{pmatrix}.
\end{gather}
By \eqref{eq:lin_combo_finite11}, we must have 
\begin{gather}\label{eq:system2}
    \begin{pmatrix}
        c_0(9) \cdots c_8(9)\\
        c_{0}(10) \cdots c_{8}(10)
\end{pmatrix} \begin{pmatrix}
       \alpha_0\\
      \vdots\\
      \alpha_8
\end{pmatrix}\equiv \begin{pmatrix}
205^2c\phi_{11}(9) \ep_9\\
229^2c\phi_{11}(10) \ep_{10} 
   \end{pmatrix}
   \pmod \ell.
\end{gather}
For each choice of $\underline \ep$ we find 
a number of primes $\ell$ which  satisfy \eqref{eq:system2}.  
We then check that none of these primes have 
$\underline \ep_{\ell} = \underline \ep$.
It follows that 
 $E(\underline \ep)=\emptyset$  for all $\underline \ep$.
\end{proof}

\section{\texorpdfstring{The case $m=7$}{m=7}}

When $k\equiv 4\pmod {12}$ we have $\dim\left(S_k(7)\right)=2(k-1)/3-1$.
There is a unique normalized cuspform $f$ in the one-dimensional space $S_4(7)$.
Its expansion is 
$f=q - q^2 - 2q^3 +\cdots$,
and it necessarily vanishes to order one at $0$.
We define 
\begin{gather}\label{eq:b7def}
F_r(z):=f\cdot \eta^{k-4}(7z)\eta^{k-4}(z)\cdot\left(\frac{\eta^4(7z)}{\eta^4(z)}\right)^{r-\frac {k-1}3}
=q^r+\cdots,\qquad 1\leq r \leq \frac{2(k-1)}3-1.
\end{gather}
It is straightforward to check that $F_r\in S_k(7)$ and that $F_r$ has order $2(k-1)/3-r$ at zero.
\begin{proof}[Proof of Main Theorem when $m=7$]
Assume that $\ell \geq 11$ and that $h_\ell \sl U_{\ell} \equiv 0 \pmod \ell$. Arguing as in the  proofs of Lemmas~\ref{lem:lin_combo_5} and ~\ref{lem:lin_combo}, we find there exist integers $\alpha_{i,\ell}$ such that 
\begin{gather}\label{eq:lincombo7}
    576\cdot \theta^{\frac{\ell+3}{2}}\,h_{\ell} \equiv \sum_{i=0}^{4} \alpha_{i,\ell} F_{r_\infty +i} \pmod \ell.
\end{gather}
As before we find that the first $7$ coefficients of $F_{r_\infty+i}/q^{r_\infty}$ appearing in \eqref{eq:lincombo7} are independent of $\ell$; for example we have 
\begin{align}
   \frac{F_{r_\infty}}{q^{r_\infty}}&= 1-5 q+4 q^2+7 q^3+27 q^4-91 q^5-35 q^6+\cdots,\\
   \frac{F_{r_\infty+1}}{q^{r_\infty}}&=q-q^2-2 q^3-7 q^4+16 q^5+2 q^6+\cdots.
\end{align}
In this case, we define 
\begin{gather}
    \underline b:= (c\phi_{7}(n)(24n-7)^2)_{0 \leq n \leq 6}, \qquad 
    \underline \ep_{\ell} :=\left(\left(\frac{24(24n-7)}\ell\right)\right)_{0 \leq n \leq 6} \in \{0, \pm 1\}^7.
\end{gather}
By \eqref{eq:h_ell_mod_ell}, $\underline b\odot\underline\ep_\ell$ gives  the first 7 coefficients of $576 \cdot\theta^{\frac{\ell+3}{2}}\,h_{\ell} $ modulo $\ell$.  For each of 576 possible $\underline \ep = (\ep_0,\dots,\ep_6) \in \{0,\pm 1\}^7$ with at most one zero entry
we find the integer vector $(\alpha_0,\alpha_1,\alpha_2,\alpha_3,\alpha_4)^T$ such that 
\begin{gather}\label{eq:sys17}
    \begin{pmatrix}
    1 & 0 & 0 & 0 & 0 \\
 -5 & 1 & 0 & 0 & 0 \\
 4 & -1 & 1 & 0 & 0 \\
 7 & -2 & 3 & 1 & 0 \\
 27 & -7 & 8 & 7 & 1 \\
\end{pmatrix}
\begin{pmatrix} 
    \alpha_0\\
    \alpha_1\\
    \alpha_2\\
    \alpha_3\\
    \alpha_4
\end{pmatrix} = \begin{pmatrix}
49\ep_0\\
  14161  \ep_1\\
    906059\ep_2\\
    14491750\ep_3\\
    135068892\ep_4
\end{pmatrix}.
\end{gather}
It follows from \eqref{eq:lincombo7} that
\begin{gather} \label{eq:sys27}
    \begin{pmatrix}
      -91 & 16 & 11 & 34 & 11 \\
 -35 & 2 & 25 & 125 & 76 \\
\end{pmatrix}
\begin{pmatrix} 
    \alpha_0\\
    \alpha_1\\
    \alpha_2\\
    \alpha_3\\
    \alpha_4
\end{pmatrix} \equiv \begin{pmatrix}
    906611769 \ep_5\\
    -4905578454 \ep_6
\end{pmatrix} \pmod \ell.
\end{gather}
Analyzing  \eqref{eq:sys27} 
as in the previous cases, we find that the only  $\underline{\ep}$
which leads to  a solution
is $\underline{\ep}_{11}$ (we describe this computation in the next section).
\end{proof}
\subsection{Example}
With 
\begin{gather}
    \underline \ep = \underline \ep_{11} = (-1,1,1,1,-1,-1,-1)
\end{gather}
 we find that \eqref{eq:sys17} is
\begin{gather}
\begin{pmatrix}
    1 & 0 & 0 & 0 & 0 \\
 -5 & 1 & 0 & 0 & 0 \\
 4 & -1 & 1 & 0 & 0 \\
 7 & -2 & 3 & 1 & 0 \\
 27 & -7 & 8 & 7 & 1 \\
\end{pmatrix} \begin{pmatrix}
    -49\\
    13916\\
    920171\\
    11759412\\
    -224647409
\end{pmatrix}
= \begin{pmatrix}
    -49\\
    14161\\
    906059\\
    14491750\\
    -135068892
\end{pmatrix}.
\end{gather}
Then \eqref{eq:sys27} becomes
\begin{gather}
\begin{pmatrix}
      -91 & 16 & 11 & 34 & 11 \\
 -35 & 2 & 25 & 125 & 76 \\
\end{pmatrix} \begin{pmatrix}
    -49\\
    13916\\
    920171\\
    11759412\\
    -224647409
\end{pmatrix} \equiv \begin{pmatrix}
    -906611769\\
    -4905578454
\end{pmatrix} \pmod \ell.
\end{gather}
This can occur only when $\ell\mid 3234$,
which shows that the only possibility is $\ell=11$.
\section{\texorpdfstring{The case $m=13$}{m=13}}
\label{sec:m=13}
Since the techniques used are similar, we will only sketch the proof of the following.

\begin{theorem}
    When $m=13$ there are no congruences of the form \eqref{eq:frob_congruences}.
\end{theorem}

To construct the necessary basis, 
 we 
let $f_4 \in S_4(13)$ be the unique cuspform which has order 2 at $\infty$ and which is fixed by the action of $W_{13}$; we
have 
\begin{gather}
f_4=q^2-2 q^3+q^5-q^6+\cdots.
\end{gather}
For  $k \equiv 4 \pmod{12}$ we show that  $S_k(13)$ has the following basis:
\begin{gather}\label{eq:basis13}
f_4\cdot \left(\eta(13z)\eta(z)\right)^{k-4}\cdot
\left(\frac{\eta^2(13z)}{\eta^2(z)}\right)^{r-2 - \frac{7(k-4)}{12}}= q^r +\cdots, \qquad 1\leq r \leq \frac{7k-4}{6}-1,
\end{gather}
and that  $\ord_0(F_r) = 7(k-4)/6 +4 -r$ for each $r$. Arguing as in the previous cases, we find that the existence of a congruence implies that $\theta^{\frac{\ell+3}{2}}\,h_{\ell} \pmod \ell$ is an $\ell$-integral linear combination of $\{F_{r_{\infty}+i}\}_{i=0}^{11}$. We then check 61440 pairs of linear systems analogous to \eqref{eq:sys1_5} and \eqref{eq:sys2_5} to rule out the existence of a congruence for $\ell > 13$.

\bibliographystyle{plain}
\bibliography{frobenius-congruences}

\begin{thebibliography}{10}

\bibitem{Ahlgren-Allen-Tang}
Scott Ahlgren, Patrick~B. Allen, and Shiang Tang.
\newblock Congruences like {A}tkin's for the partition function.
\newblock {\em Trans. Amer. Math. Soc. Ser. B}, 9:1044--1064, 2022.

\bibitem{AAD:frob}
Scott Ahlgren, Nickolas Andersen, and Robert Dicks.
\newblock Congruences like {A}tkin's for generalized {F}robenius partitions.
\newblock {\em https://arxiv.org/abs/2504.03954}, preprint.

\bibitem{ahlgren-beckwith-raum}
Scott Ahlgren, Olivia Beckwith, and Martin Raum.
\newblock Scarcity of congruences for the partition function.
\newblock {\em Amer. J. Math.}, 145(5):1509--1548, 2023.

\bibitem{Ahlgren-Boylan}
Scott Ahlgren and Matthew Boylan.
\newblock Arithmetic properties of the partition function.
\newblock {\em Invent. Math.}, 153(3):487--502, 2003.

\bibitem{githubfile}
Scott Ahlgren and Cruz Castillo.
\newblock
  \url{https://github.com/cruzcstll99/Arithmetic-Properties-of-Generalized-Frobenius-Partitions},
  2025.

\bibitem{Andrews}
George~E. Andrews.
\newblock Generalized {F}robenius partitions.
\newblock {\em Mem. Amer. Math. Soc.}, 49(301):iv+44, 1984.

\bibitem{Chan-Wang-Yang}
Heng~Huat Chan, Liuquan Wang, and Yifan Yang.
\newblock Modular forms and {$k$}-colored generalized {F}robenius partitions.
\newblock {\em Trans. Amer. Math. Soc.}, 371(3):2159--2205, 2019.

\bibitem{DeligneRapoport}
P.~Deligne and M.~Rapoport.
\newblock Les sch\'{e}mas de modules de courbes elliptiques.
\newblock In {\em Modular functions of one variable, {II} ({P}roc. {I}nternat.
  {S}ummer {S}chool, {U}niv. {A}ntwerp, {A}ntwerp, 1972)}, Lecture Notes in
  Math., Vol. 349, pages 143--316. Springer, Berlin-New York, 1973.

\bibitem{Diamond-Shurman}
Fred Diamond and Jerry Shurman.
\newblock {\em A first course in modular forms}, volume 228 of {\em Graduate
  Texts in Mathematics}.
\newblock Springer-Verlag, New York, 2005.

\bibitem{Gross}
Benedict~H. Gross.
\newblock A tameness criterion for {G}alois representations associated to
  modular forms (mod {$p$}).
\newblock {\em Duke Math. J.}, 61(2):445--517, 1990.

\bibitem{Jochnowitz}
Naomi Jochnowitz.
\newblock A study of the local components of the {H}ecke algebra mod {$l$}.
\newblock {\em Trans. Amer. Math. Soc.}, 270(1):253--267, 1982.

\bibitem{KimingOlsson}
Ian Kiming and J\o rn~B. Olsson.
\newblock Congruences like {R}amanujan's for powers of the partition function.
\newblock {\em Arch. Math. (Basel)}, 59(4):348--360, 1992.

\bibitem{Koblitz}
Neal Koblitz.
\newblock {\em Introduction to elliptic curves and modular forms}, volume~97 of
  {\em Graduate Texts in Mathematics}.
\newblock Springer-Verlag, New York, second edition, 1993.

\bibitem{ono_web}
Ken Ono.
\newblock {\em The web of modularity: arithmetic of the coefficients of modular
  forms and {$q$}-series}, volume 102 of {\em CBMS Regional Conference Series
  in Mathematics}.
\newblock Conference Board of the Mathematical Sciences, Washington, DC; by the
  American Mathematical Society, Providence, RI, 2004.

\bibitem{weston}
Tom Weston.
\newblock The modular curves ${X}_0(11)$ and ${X}_1(11)$ (unpublished notes).
\newblock {\em \\ https://swc-math.github.io/notes/files/01Weston1.pdf}.

\end{thebibliography}
\end{document}